\documentclass[12pt]{amsart}

\usepackage{amsmath, amssymb, amscd, newlfont}
\usepackage{ucs}
\usepackage[cp1250]{inputenc}
\usepackage{makecell}
\usepackage{caption}

\newtheorem{theorem}{Theorem}[section]

\newtheorem{lemma}[theorem]{Lemma}

\title
[Eta-quotients and models for $X_0(N)$]
{Eta-quotients and embeddings of $X_0(N)$ in the projective plane}

\author{Iva Kodrnja }
\address{
Faculty of Civil Engineering, 
University of Zagreb,
Ka\v ci\' ceva 26, 10000 Zagreb,
Croatia}
 \email{ikodrnja@grad.hr}
\thanks{The  author acknowledges Croatian Science Foundation grant no. 9364.}

\begin{document}
\begin{abstract}
In this paper we find projective plane models of $X_0(N)$ by constructing maps from $X_0(N)$ to the projective plane using modular forms, as presented in \cite{Muic2}. We use eta-quotients of weight $12$. We find those eta-quotients of weight $12$ which have maximal order of zero at the cusp $\infty$.
\end{abstract}

\maketitle

\section{Introduction}
\label{uvod}
Let $\Gamma_0(N)$ be the congruence subgroup $$\Gamma_0(N)=\left\lbrace \gamma\in SL_2(\mathbb{Z}):\gamma\equiv\begin{pmatrix} * & * \\ 0 & *\end{pmatrix}\pmod{N}\right\rbrace.$$
This group acts on the extended complex upper half plane $\mathbb{H}^*=\mathbb{H}\cup\mathbb{Q}\cup\left\lbrace \infty\right\rbrace $, with $\mathbb{H}=\left\lbrace z\in\mathbb{C}:\Im z>0\right\rbrace$,
by linear fractional transformations $$\gamma.z=\frac{az+b}{cz+d},\quad \gamma=\begin{pmatrix} a & b \\ c & d\end{pmatrix}\in \Gamma_0(N).$$
The quotient space of this group action is called a modular curve, and we denote it by $X_0(N)$ i.e.,
$$X_0(N)=\Gamma_0(N)\setminus \mathbb{H}^*.$$
The set $X_0(N)$ can be endowed with a complex structure and it is a compact Riemann surface. It is of interest in number theory to embed this Riemann surface into the projective or affine space and find its defining equations. The curve that is an image of such an embedding is called a \textit{model} of the modular curve $X_0(N)$.

The field of rational functions over $\mathbb{C}$ of $X_0(N)$ is generated by $j$ and $j(N\cdot)$. The minimal polynomial of $j(N\cdot)$ over $\mathbb{C}(j)$ is called the classical modular polynomial and it gives the canonical plane model for $X_0(N)$, \cite{shi}. However, this polynomial is hard to compute and has enormous coefficients so it is not of practical use. That is why people search for different models for the modular curve $X_0(N)$,\cite{emb},\cite{Muic1},\cite{Muic2},\cite{mshi},\cite{yy}. One method uses the canonical embedding of Riemann surfaces in the projective space, \cite{emb}. Using the connection of modular forms on $\Gamma_0(N)$ (or in general on any Fuchsian group of the first kind) with the differentials on $X_0(N)$, G. Mui\' c has searched for models of $X_0(N)$ by constructing maps into the projective space using modular forms of arbitrary weight, \cite{Muic1},\cite{Muic2}. We use this method to construct models of $X_0(N)$ into the projective plane $\mathbb{P}^2$, as in \cite{Muic2}, using eta-quotients.

The Dedekind eta-function is defined by the infinite product
\begin{equation*}
\eta(z)=q^{\frac{1}{24}}\prod\limits_{n=1}^\infty (1-q^n),\quad z\in\mathbb{H},
\end{equation*} 
where $q=q(z):=e^{2\pi i z}$. This function is holomorphic function on $\mathbb{H}$ with no zeroes on $\mathbb{H}$.
An \textit{eta-quotient} of level $N$ is a finite product of the form
\begin{equation*}
f(z)=\prod\limits_{\delta | N}\eta(\delta z)^{r_\delta},\quad r_\delta\in\mathbb{Z}.
\end{equation*} 
 These functions have some beautiful properties. They have integral Fourier expansions at the cusp $\infty$, using their modular transformation properties, we can calculate their Fourier expansions at other cusps.
  There is a formula for the order of an eta-quotient at each cusp so we can easily calculate their divisors.

There are always two eta-quotients that are modular forms of weight $12$ for $\Gamma_0(N)$ for every $N$, Ramanujan delta function $\Delta$ and its rescaling $\Delta(N\cdot)$. We search for a third function so that the map into the projective plane defined with these three modular forms is a birational equivalence.
For a prime $p$, we calculate the degree of the map defined with $\Delta,\Delta(p\cdot)$ and $\eta^{12}(z)\eta^{12}(pz)\in M_{12}(\Gamma_0(p))$.

In \cite{choi} it is proved that the unique normalised modular form of weight $12$ for $\Gamma_0(N)$ with maximal order of vanishing at the cups $\infty$ is an eta-quotient when the genus of $\Gamma_0(N)$ equals zero. We generalise this claim, and find those numbers $N$ for which the unique normalised modular form of weight $12$ for $\Gamma_0(N)$ with maximal order of vanishing at the cups $\infty$ is an eta-quotient. Using these functions and functions $\Delta$ and $\Delta(N\cdot)$ we construct maps from $X_0(N)$ to the projective plane and prove that in some cases this map is a birational equivalence.

\section{Preliminaries}
The group $\Gamma_0(N)$ is a modular group i.e., a subgroup of $SL_2(\mathbb{Z})$ of finite index which is equal to the value of the Dedekind Psi function 
\begin{equation}\label{indeks0}
\left[ SL_2(\mathbb{Z}):\Gamma_0(N)\right]=N\prod\limits_{p|N}(1+1/p)=\Psi(N).
\end{equation}

The group $\Gamma_0(N)$ has $\sum\limits_{0<d|N}\phi((d,N/d))$ cusps. As a set of representatives of inequivalent cusps we can take the set 
\begin{equation}\label{ku0}
\mathcal{C}_N=\left\lbrace\frac{a}{d}:d|N,(a,d)=1, a\in(\mathbb{Z}/k\mathbb{Z})^*\text{ za } k=(d,N/d) \right\rbrace .
\end{equation} 

There are $\Phi((d,N/d))$ cusps with denominator $d$, for each divisor $d$ of $N$. There are always two cusps for $\Gamma_0(N)$, for every $N$ - one cusps with denominator $1$ which we denote $\mathfrak{0}$ and one cusp with denominator $N$ which we denote as cusp $\infty$. 


Every eta-quotient is a holomorphic function of the upper half plane with no zeroes on the upper half plane. The most famous example is the Ramanujan delta function,
$$\Delta(z)=\eta(z)^{24}=\sum\limits_{n=1}^\infty \tau(n)q^n,$$
which is a cusp form of weight $12$ on $SL_2(\mathbb{Z})$. We have $$\Delta,\Delta(N\cdot)\in M_{12}(\Gamma_0(N)), \text{ for all } N,$$
so there are always two eta-quotients of weight $12$ on $\Gamma_0(N)$. 

Ligozat in \cite{ligozat}, Proposition 3.2.8 proved the formula for the order of vanishing of an eta-quotient at the cusps with denominator $d$:
\begin{equation}
\frac{N}{24}\sum\limits_{\delta|N} \frac{(\delta,d)^2r_\delta}{(N,d^2)\delta}.
\end{equation}

Using modular transformation properties of the Dedekind eta-function (\cite{apostol}, Theorem 3.4) we can deduce the modular transformation properties of eta-quotients. This result in various forms can be found in \cite{ono} Theorem 1.64, \cite{ligozat} Proposition 3.2.1, \cite{newman1} Theorem 1. 

For a divisor $\delta$ of $N$, we denote by $\delta'$ the number $\delta'\delta=N$.
\begin{theorem} \label{uvjeti}
If 
\begin{itemize}
\item[(1)] $\sum\limits_{\delta | N}\delta r_\delta\equiv 0\pmod{24}$
\item[(2)] $\sum\limits_{\delta | N}\delta' r_\delta\equiv 0\pmod{24}$
\item[(3)] $\prod\limits_{\delta | N} \delta'^{r_\delta}$ is a square of a rational number,
\end{itemize}
then $\prod\limits_{\delta | N}\eta(\delta z)^{r_\delta}$ is a weakly modular function on $\Gamma_0(N)$ of weight $k=\frac{1}{2}\sum_\delta r_\delta$.

If we require that 
\begin{itemize}
\item[(4)] $\frac{N}{24}\sum\limits_{\delta|N} \frac{(\delta,d)^2r_\delta}{(N,d^2)\delta}\geq 0$, for all divisors $d$ of $N$, 
\end{itemize}
then $\prod\limits_{\delta | N}\eta(\delta z)^{r_\delta}$ is a modular form on $\Gamma_0(N)$ of weight $k=\frac{1}{2}\sum_\delta r_\delta$.
\end{theorem} 
We can write the conditions $(4)$ in a matrix form. If we denote  $$A_N=\left(\frac{(\delta,d)^2r_\delta}{(N,d^2)\delta}\right)_{d,\delta}$$ for $d,\delta$ which are divisors of $N$. This is a square matrix of size $\sigma_0(N)$. This matrix is called the \textit{order matrix}, \cite{batt1}.
If we write the exponents $r_\delta$ of an eta-quotient $\prod_{\delta|N}\eta(\delta z)^{r_\delta}$ as a column vector in increasing order $$\textbf{r}=(r_\delta)_\delta,$$
then the condition $(4)$ is equivalent to $$A_N\textbf{r}\geq 0.$$

 The matrix $A_N$ is invertible over $\mathbb{Q}$ for every $N$ and if the prime factorisation of $N$ is $N=p_1^{n_1}\cdots p_s^{n_s}$, then the matrix $A_N$ is the Kronecker product (\cite{batt}, Proposition 1.41) $$A_N=A_{p_1^{n_1}}\bigotimes \cdots\bigotimes A_{p_s^{n_s}}.$$

\section{Maps $X_0(N)\to\mathbb{P}^2$ via modular forms}
Let $k\geq 2$ be an even integer such that $\dim M_k(\Gamma_0(N))\geq 3$. Let $f,g,h\in M_k(\Gamma_0(N))$ be three linearly independent modular forms. We construct the holomorphic map $\varphi:X_0(N)\to\mathbb{P}^2$ which is uniquely determined by being initially defined by 

$$\varphi(\mathfrak{a}_z)=(f(z):g(z):h(z))$$ 
on the complement of a finite set of $\Gamma_0(N)$-orbits in $X_0(N)$ of common zeroes of $f$,$g$ and $h$.
Every compact Riemann surface can be observed as a smooth irreducible projective curve over $\mathbb{C}$, and functions $g/f$ and $h/f$ are rational functions on $X_0(N)$. Thus, the map $\varphi$ is actually a rational map $$\mathfrak{a}_z\mapsto (1:g(z)/f(z):h(z)/f(z)).$$

Since $X_0(N)$ is smooth, the map is regular and the image is an irreducible curve in $\mathbb{P}^2$. The image is not constant because the functions $f,g$ and $h$ are linearly independent.

The image of this map is in most cases singular projective plane curve, which we denote by $\mathcal{C}(f,g,h)$. 
The field of rational functions of the curve $\mathcal{C}(f,g,h)$, denoted by $\mathbb{C}(\mathcal{C}(f,g,h))$ is isomorphic to a subfield of $\mathbb{C}(X_0(N))$, the field of rational functions of the modular curve $X_0(N)$, by the map $\varphi$. By definition, the degree of the map $\varphi$ which will be denoted by $$d(f,g,h)$$ is equal to the degree of the field extension $$\mathbb{C}(\mathcal{C}(f,g,h))\subset \mathbb{C}(X_0(N)).$$ 

There is a simple criterion for the map to be of degree $1$ i.e., to be a birational equivalence. It can be found in \cite{yy}, Lemma 1 or \cite{Muic}, Lemma 5-2. 
\begin{lemma}\label{stupnjevipolovi}
The degree of the maps $\varphi$, $d(f,g,h)$ divides the numbers \begin{equation}\label{degdivpol}
\deg(\text{div}_\infty(g/f))\quad \text{ and } \quad \deg(\text{div}_\infty(h/f)).\end{equation}
The sufficient condition for the map $\varphi$ to be a birational equivalence is that these numbers are relatively prime i.e.,
\begin{equation}
(\deg(\text{div}_\infty(g/f)), \deg(\text{div}_\infty(h/f)))=0.\end{equation}
\end{lemma}

\begin{proof}
The field of rational functions of $\mathcal{C}(f,g,h)$ is isomorphic to the subfield of $\mathbb{C}(X_0(N))$ generated over $\mathbb{C}$ with $g/f$ and $h/f$. We have the following inclusions of fields 
\begin{align*}
&\mathbb{C}(g/f)\subseteq\mathbb{C}(g/f,h/f)\subseteq\mathbb{C}(X(\Gamma))\\
&\mathbb{C}(h/f)\subseteq\mathbb{C}(g/f,h/f)\subseteq\mathbb{C}(X(\Gamma)).
\end{align*}
From the theory of compact Riemann surfaces we know that for a meromorphic function $f$ on a compact Riemann surface $X$ we have 
\begin{equation*}
\left[\mathbb{C}(X):\mathbb{C}(f) \right]=\deg\text{div}_\infty(f). 
\end{equation*}
The claim follows.
\end{proof}

Let $f\in M_k(\Gamma_0(N))$, $f\neq 0$. For each $\mathfrak{a}\in X_0(N)$, we can define the multiplicity $\nu_\mathfrak{a}(f)$ of $f$ at $\mathfrak{a}$ (see \cite{Muic},\S 4, Lemma 4-1). This number is rational when $\mathfrak{a}$ is an elliptic point, otherwise it is integral.
Then we may define the divisor of $f$ as $$\text{div}(f)=\sum\limits_{\mathfrak{a}\in X_0(N)}\nu_\mathfrak{a}(f)\mathfrak{a}.$$ 
The degree of the divisor of a modular form in $M_k(\Gamma_0(N))$ equals \begin{equation*}
\deg(\textrm{div}(f))=k(g-1)+\frac{k}{2}\left(t+\sum\limits_{\textbf{a}\in X_0(N)\textrm{, elliptic}} (1-1/e_\mathfrak{a}) \right),
 \end{equation*} 
where $g$ is the genus of $\Gamma_0(N)$, $t$ is the number of inequivalent cusps and $e_\mathfrak{a}$ is the index of ramification at an elliptic point $\mathfrak{a}$. 
By substracting the possible non-integer parts of multiplicities at elliptic points, we obtain an integral divisor $D_f$ attached to the modular  form $f\in M_k(\Gamma_0(N))$ and it has the following degree
\begin{equation}
\deg D_f=\dim M_k(\Gamma_0(N))+g(\Gamma_0(N))-1.
\end{equation}   

When $k=12$, using the formula for the genus (\cite{Miyake}, Theorem 4.2.11) and for the dimension of $M_{12}(\Gamma_0(N))$ (\cite{Miyake}, Theorem 2.5.2) we get that
\begin{equation}
\label{racun}
\dim M_{12}(\Gamma_0(N))+g(\Gamma_0(N))-1=\left[SL_2(\mathbb{Z}):\Gamma_0(N)\right].
\end{equation}

In \cite{Muic2}, Corollary 1-5, the following fact is proved which we now state for the case of the modular group $\Gamma_0(N)$:
\begin{theorem}\label{kriterij}
Assume $k\geq 2$ is an even integer such that $\dim M_k(\Gamma_0(N))\geq 3$. Let $f,g,h\in M_k(\Gamma_0(N))$ be linearly independent. Then, we have the following: 
\begin{equation}\label{formulazastupanj}
 d(f,g,h)\deg C(f,g,h)=\dim M_k(\Gamma_0(N))+g(\Gamma_0(N))-1-\sum\limits_{\mathfrak{a}\in X(\Gamma)} \min (D_f(\mathfrak{a}),D_g(\mathfrak{a}),D_h(\mathfrak{a})).
 \end{equation}

\end{theorem}

In my doctoral thesis, I have developed an algorithm that calculates the degree of the curve $\mathcal{C}(f,g,h)$ and using the formula $(\ref{formulazastupanj})$ we can calculate the degree of the map.

\section{Eta-quotients of weight $12$ and models of $X_0(N)$}

For a prime number $p$, modular curve $X_0(p)$ has two cusps which we denote by $\mathfrak{a}_\infty$ and $\mathfrak{a}_{\mathfrak{0}}$. For $p>2$, there are always the following three eta-quotients in $M_{12}(\Gamma_0(p))$:
$$\eta(z)^{24}=\Delta(z),\quad \eta(pz)^{24}=\Delta_p(z) \quad\text{ and } \quad\eta(z)^{12}\eta(pz)^{12}.$$
 Their divisors with regard to $\Gamma_0(p)$ are as follows, 
\begin{align}\label{div}
\text{div}(\Delta(z))&=p\mathfrak{a}_{\mathfrak{0}}+\mathfrak{a}_\infty, \nonumber \\  
\text{div}(\Delta_p(z))&=\mathfrak{a}_{\mathfrak{0}}+p\mathfrak{a}_\infty,\\
\text{div}(\eta(z)^{12}\eta(pz)^{12})&=\frac{p+1}{2}\mathfrak{a}_{\mathfrak{0}}+\frac{p+1}{2}\mathfrak{a}_\infty ,\nonumber
\end{align}

Their Fourier exapansion at the cusp $\infty$ are 
\begin{align*}
\Delta(z)&=q-24q^2+252q^3-1427q^4+...\\
\Delta_p(z)&=q^p-24q^{2p}+252q^{3p}-1427q^{4p}+...\\
\eta(z)^{12}\eta(pz)^{12}&=q^{\frac{p+1}{2}}-12q^{\frac{p+1}{2}+1}+54q^{\frac{p+1}{2}+2}-88q^{\frac{p+1}{2}+3}-99q^{\frac{p+1}{2}+4}+...
\end{align*}
 and we see that these forms are linearly independent.

\begin{theorem}
Let $p>2$ be a prime number.
The map $\varphi:X_0(p)\to\mathbb{P}^2$ defined as
$$\varphi(\mathfrak{a}_z)=(\Delta(z):\eta(z)^{12}\eta(pz)^{12}:\Delta_p(z))$$
has degree $\frac{p-1}{2}$. 
\end{theorem}
\begin{proof}
 We calculate $$(\eta(z)^{12}\eta(pz)^{12})^2-\Delta(z)\Delta_p(z)=\eta^{24}(z)\eta^{24}(pz)-\eta^{24}(z)\eta^{24}(pz)=0$$
 and conclude that the image of the map $\varphi$ is a projective curve of degree $2$ whose defining polynomial is $x_1^2-x_0x_2$.

Minimum of the divisors $(\ref{div})$ from the right side of the formula $(\ref{formulazastupanj})$ from the Theorem \ref{kriterij} equals $2$.
Furthemore, we calculate 
$$\dim M_{12}(\Gamma_0(p))+g(\Gamma_0(p))-1=p+1.$$
The right side of the formula $(\ref{formulazastupanj})$ from Theorem \ref{kriterij} equals $p-1$.

 From the formula $(\ref{formulazastupanj})$ it follows that the degree of the map $\varphi$ equals $\frac{p-1}{2}$. 
\end{proof}

In \cite{choi}, Lemma 2.1, it is proved that the unique normalised modular form with the maximal order of zero at the cusp $\infty$ is an eta-quotient for $\Gamma_0(N)$ when the genus of $\Gamma_0(N)$ is zero. We generalize this claim and find those eta-quotients of weight $12$ that have the only zero at the cusp $\infty$. The order of the zero is maximal. In \cite{rousewebb}, Lemma 14, it is proved that for every cusp of $\Gamma_0(N)$ there is an eta-quotient of some weight such that it vanishes only at that cusp.
\begin{theorem}\label{NN}
Unique normalised modular form of weight $12$ for $\Gamma_0(N)$ which has a zero of maximal order at the cusp $\infty$ is an eta-quotient if $N$ belongs to one of the following sets: 
\begin{align*}
& S_1=\left\lbrace p^n : p\in\left\lbrace 2,3,5,7,13 \right\rbrace , n\geq 1\right\rbrace \\
&S_2=\left\lbrace p_1^{n_1}p_2^{n_2}: p_1\in\left\lbrace 2,3,5 \right\rbrace,p_2\in\left\lbrace 3,5,7,13 \right\rbrace,p_1p_2<40,n_1,n_2\geq 1\right\rbrace \\
&S_3=\left\lbrace  p_1^{n_1}p_2^{n_2}p_3^{n_3}: p_1=2,p_2\in\left\lbrace 3,5 \right\rbrace,p_3\in\left\lbrace 5,7,13 \right\rbrace,p_2+p_3<17,n_1,n_2,n_3\geq 1\right\rbrace. 
\end{align*}

Then, that eta-quotient is given by:
\begin{itemize}
\item[(i)] $$\frac{\eta(p^{n}z)^{pr}}{\eta(p^{n-1}z)^r}$$
for $N\in S_1$, with $r=\frac{24}{p-1}$ , 
\item[(ii)] $$\frac{\eta(p_1^{n_1-1}p_2^{n_2-1}z)^{r}\eta(p_1^{n_1}p_2^{n_2}z)^{p_1p_2r}}{\eta(p_1^{n_1}p_2^{n_2-1}z)^{p_1r}\eta(p_1^{n_1-1}p_2^{n_2}z)^{p_2r}}$$
 for $N\in S_2$, with $$ r=\frac{24}{(p_1-1)(p_2-1)}$$
 \item[(iii)] $$\frac{\eta(p_1^{n_1}p_2^{n_2-1}p_3^{n_3-1}z)^{p_1r}\eta(p_1^{n_1-1}p_2^{n_2}p_3^{n_3-1}z)^{p_2r}\eta(p_1^{n_1-1}p_2^{n_2-1}p_3^{n_3}z)^{p_3r}\eta(p_1^{n_1}p_2^{n_2}p_3^{n_3}z)^{p_1p_2p_3r}}{\eta(p_1^{n_1-1}p_2^{n_2-1}p_3^{n_3-1}z)^{r}\eta(p_1^{n_1}p_2^{n_2}p_3^{n_3-1}z)^{p_1p_2r}\eta(p_1^{n_1}p_2^{n_2-1}p_3^{n_3}z)^{p_1p_3r}\eta(p_1^{n_1-1}p_2^{n_2}p_3^{n_3}z)^{p_2p_3r}}$$
  for $N\in S_3$, with $$r=\frac{24}{(p_1-1)(p_2-1)(p_3-1)}.$$ 
\end{itemize}
\end{theorem}

\begin{proof}
We check when an eta-quotient $\prod_{\delta|N} \eta(\delta z)^{r_\delta}$ has a zero of maximal order at the cusp $\infty$.
 We write coefficients $r_\delta$ of this eta-quotient as a column vector $\textbf{r}_\delta$ (arranged by size). For $N$ we look at the prime factorisation, $N=p_1^{n_1}\dots p_s^{n_s}$, $p_1<p_2<\dots<p_s$. The order matrix $A_N$ is a Kronecker product of matrices for prime divisors of $N$,
$A_N=A_{p_1^{n_1}}\bigotimes\cdots\bigotimes A_{p_s^{n_s}}$. Coefficients of an eta-quotient with a zero of maximal order at the cusp $\infty$ are the solutions of the system:

\begin{equation*}
A_{N}\textbf{r}_\delta=\begin{pmatrix}0\\\vdots\\0\\p_1^{n_1-1}\cdots p_s^{s_1-1}(p_1+1)\cdots(p_s+1)\end{pmatrix},
\end{equation*}
i.e., vector $\textbf{r}_\delta$ is the product 
\begin{equation}\label{umn}A_N^{-1}\begin{pmatrix}0\\\vdots\\0\\p_1^{n_1-1}\cdots p_s^{s_1-1}(p_1+1)\cdots(p_s+1)\end{pmatrix}.\end{equation}

The inverse of the order matrix $A_N^{-1}$ is the Kronecker product of inverses of matrices $A_{p_i^{n_i}}$.
For a prime number $p$, matrix $A_{p^n}$ is a square matrix of size $(n+1)\times(n+1)$ given by
\begin{equation*}
24\cdot A_{p^n}(i,j)=\left\lbrace \begin{array}{cc}
\frac{p^n}{p^{\min\left\lbrace i,n-i\right\rbrace }}, & i=j\\
\frac{p^n}{p^{j-i}p^{\min\left\lbrace i,n-i\right\rbrace }},&i<j\\
\frac{p^n}{p^{i-j}p^{\min\left\lbrace i,n-i\right\rbrace }},&i>j\end{array}\right.
\end{equation*} 

and its inverse is given by
 $$\frac{1}{24}A^{-1}_{p^n}=\frac{1}{p^{n-1}(p^2-1)}\begin{pmatrix}
p&-p&&&&&\\-1&p^2+1&-p^2&&&\Large{0}&\\&-p&p(p^2+1)&-p^3&&&\\&&\ddots&\ddots&\ddots&&\\&\Large{0}&&&-p^2&p^2+1&-1\\&&&&&-p&p\end{pmatrix}$$

 The last column of $A_N$ has $2^s$ elements different from $0$. First non-zero entry is $(-1)^s$ and from $(\ref{umn})$ we get the equation 
$$r_{p_1^{n_1-1}\cdots p_s^{n_s-1}}=\frac{(-1)^s 24}{(p_1-1)\cdots(p_s-1)}.$$

The condition $r_{p_1^{n_1-1}\cdots p_s^{n_s-1}}\in\mathbb{Z}$ implies $s<4$.

When $s=1$, from $(\ref{umn})$ we have equations 
\begin{align*}
&r_1=r_p=\dots=r_{p^{n-2}}=0\\
&r_{p^{n-1}}=\frac{-24}{p-1}\\
&r_{p^n}=\frac{24p}{p-1}
\end{align*}
The condition $r_{p^{n-1}}\in\mathbb{Z}$ implies $p=2,3,5,7,13$ i.e., $N\in S_1$.

For $s=2$ we have equations
 \begin{align*}
&r_{p_1^ip_2^j}=0,\text{ for } i<n_1-1 \text{ or } j<n_2-1\\
&r_{p_1^{n_1-1}p_2^{n_2-1}}=\frac{24}{(p_1-1)(p_2-1)}=x\\
&r_{p_1^{n_1-1}p_2^{n_2}}=-p_2x\\
&r_{p_1^{n_1}p_2^{n_2-1}}=-p_1x\\
&r_{p_1^{n_1}p_2^{n_2}}=p_1p_2x\\
 \end{align*}
The condition $r_{p_1^{n_1-1}p_2^{n_2-1}}\in\mathbb{Z}$ implies that $N$ must belong to $S_2$.

For $s=3$ we have
\begin{align*}
&r_{p_1^ip_2^jp_3^k}=0,\text{ for } i<n_1-1 \text{ or } j<n_2-1 \text{ or } k<n_3-1\\
&r_{p_1^{n_1-1}p_2^{n_2-1}p_3^{n_3-1}}=\frac{-24}{(p_1-1)(p_2-1)(p_3-1)}=x\\
&r_{p_1^{n_1}p_2^{n_2-1}p_3^{n_3-1}}=-p_1x   ,\quad r_{p_1^{n_1-1}p_2^{n_2}p_3^{n_3-1}}=-p_2x,\quad  r_{p_1^{n_1-1}p_2^{n_2-1}p_3^{n_3}}=-p_3x\\
&r_{p_1^{n_1-1}p_2^{n_2}p_3^{n_3}}=p_2p_3x, \quad  r_{p_1^{n_1}p_2^{n_2-1}p_3^{n_3}}=p_1p_3x,\quad r_{p_1^{n_1}p_2^{n_2}p_3^{n_3-1}}=p_1p_2x\\
&r_{p_1^{n_1}p_2^{n_2}p_3^{n_3}}=-p_1p_2p_3x
\end{align*}

From the condition $r_{p_1^{n_1-1}p_2^{n_2-1}p_3^{n_3-1}}\in\mathbb{Z}$ we have the following possible values $$(p_1,p_2,p_3)=(2,3,5),(2,3,7),(2,3,13),(2,5,7)$$
so $N$ must belong to $S_3$.

Now we check that these functions are modular forms by checking the conditions of Theorem \ref{uvjeti}.

Let $N=p^n\in S_1$. 
Denote $r=\frac{24}{p-1}$ and $\Delta_{p^n,12}(z)=\frac{\eta(p^{n}z)^{24+r}}{\eta(p^{n-1}z)^r}$. We check the conditions of Theorem \ref{uvjeti}.  We have 
$$
\sum\limits_{\delta|N}\delta r_\delta=-p^{n-1}r+p^n(24+r)=p^{n-1}((p-1)r+24p)=24(p+1)p^{n-1}.
$$
Order of $\Delta_{p^n,12}$ at the cusp $\infty$ is $1/24\sum_\delta \delta r_\delta=(p+1)p^n$ and we see that this function has the maximal order at the cusp $\infty$ by the valence theorem since the index of $\Gamma_0(p^n)$ in $SL_2(\mathbb{Z})$ equals $(p+1)p^{n-1}$.
The other conditions are also satisfied (order at all other cusps is equal to zero):
\begin{align*}
&\sum\limits_{\delta|N}\delta'r_\delta=-pr+24+r=\frac{24}{p-1}(1-p)+24=0 \small{\text{ ( order at the cusp $\mathfrak{0}$ equals $0$)}}\\
&\prod\limits_{\delta|N}\delta^{r_\delta}=p^{-r(n-1)}p^{n(24-r)}=p^{-2rn+r+24n} \small{\text{ is a square of an integer because $r$ is even}}\\
&\sum\limits_{\delta|N}r_\delta=-r+24+r=2\cdot 12
\end{align*}

This proves that $\Delta_{p^n,12}\in M_{12}(\Gamma_0(p^n))$ has maximal order of zero at the cusp $\infty$.

For $N=p_1^{n_1}p_2^{n_2}$ we denote $r=\frac{24}{(p_1-1)(p_2-1)}$ and for the function $$\Delta_{p_1^{n_1}p_2^{n_2},12}(z)=\frac{\eta(p_1^{n_1-1}p_2^{n_2-1}z)^{r}\eta(p_1^{n_1}p_2^{n_2}z)^{p_1p_2r}}{\eta(p_1^{n_1}p_2^{n_2-1}z)^{p_1r}\eta(p_1^{n_1-1}p_2^{n_2}z)^{p_2r}}$$ we check the conditions of the Theorem \ref{uvjeti}:
\begin{align*}
\sum\limits_{\delta|N}\delta r_\delta&=p_1^{n_1-1}p_2^{n_2-1}r(1+p_1^2p_2^2-p_1^2-p_2^2)\\
                                    &=p_1^{n_1-1}p_2^{n_2-1}\frac{24}{(p_1-1)(p_2-1)}(p_1+1)(p_1-1)(p_2+1)(p_2-1)\\
                                    &=24p_1^{n_1-1}p_2^{n_2-1}(p_1+1)(p_2+1).
\end{align*}
Order at the cusp $\infty$ equals $1/24\sum_\delta \delta r_\delta=p_1^{n_1-1}p_2^{n_2-1}(p_1+1)(p_2+1)$ which is maximal order of the zero by the valence theorem since the index of $\Gamma_0(p_1^{n_1}p_2^{n_2})$ is $p_1^{n_1-1}p_2^{n_2-1}(p_1+1)(p_2+1)$.
Order at all other cusps must be equal to zero. The remaining conditions are
\begin{align*}
&\sum\limits_{\delta|N}\delta'r_\delta=p_1p_2r+p_1p_2r-p_1p_2r-p_1p_2r=0\\
&\sum\limits_{\delta|N}r_\delta=r(1+p_1p_2-p_1-p_2)=\frac{24}{(p_1-1)(p_2-1)} (p_1-1)(p_2-1)=24.
\end{align*}

\begin{align*}
\prod\limits_{\delta|N} \delta'^{r_\delta}&=(p_1p_2)^rp_2^{-p_1r}p_1^{-p_2r}=p_1^{r(1-p_2)}p_2^{r(1-p_1)}\\
                                          &=p_1^{\frac{-24}{p_1-1}}p_2^{\frac{-24}{p_2-1}}.
\end{align*}
We have $p_1,p_2\in\left\lbrace 2,3,5,7,13\right\rbrace$, so the numebers  $\frac{-24}{p_1-1}$ and $\frac{-24}{p_2-1}$ are even, and we see that $\prod\limits_{\delta|N} \delta'^{r_\delta}$ is a square of a rational number.

Thus we proved that $\Delta_{p_1^{n_1}p_2^{n_2},12}(z)\in M_{12}(\Gamma_0(p_1^{n_1}p_2^{n_2}))$ has a zero of maximal order at the cusp $\infty$.

Assume $N\in S_3$. Denote $r=\frac{24}{(p_1-1)(p_2-1)(p_3-1)}$ and $$\Delta_{p_1^{n_1}p_2^{n_2}p_3^{n_3},12}=\frac{\eta(p_1^{n_1}p_2^{n_2-1}p_3^{n_3-1}z)^{p_1r}\eta(p_1^{n_1-1}p_2^{n_2}p_3^{n_3-1}z)^{p_2r}\eta(p_1^{n_1-1}p_2^{n_2-1}p_3^{n_3}z)^{p_3r}\eta(p_1^{n_1}p_2^{n_2}p_3^{n_3}z)^{p_1p_2p_3r}}{\eta(p_1^{n_1-1}p_2^{n_2-1}p_3^{n_3-1}z)^{r}\eta(p_1^{n_1}p_2^{n_2}p_3^{n_3-1}z)^{p_1p_2r}\eta(p_1^{n_1}p_2^{n_2-1}p_3^{n_3}z)^{p_1p_3r}\eta(p_1^{n_1-1}p_2^{n_2}p_3^{n_3}z)^{p_2p_3r}}.$$ We check the conditions of Theorem \ref{uvjeti}.

\begin{align*}
\sum\limits_{\delta|N}\delta r_\delta&=p_1^{n_1-1}p_2^{n_2-1}p_3^{n_3-1}r(p_1^2+p_2^2+p_3^2+p_1^2p_2^2p_3^2-1-p_1^2p_2^2-p_1^2p_3^2-p_2^2p_3^2)\\
                                     &=p_1^{n_1-1}p_2^{n_2-1}p_3^{n_3-1}\frac{24}{(p_1-1)(p_2-1)(p_3-1)} (p_1-1)(p_1+1)(p_2-1)(p_2+1)(p_3-1)(p_3+1)\\
                                     &=24p_1^{n_1-1}p_2^{n_2-1}p_3^{n_3-1}(p_1+1)(p_2+1)(p_3+1)
\end{align*}

Order at the cusp $\infty$ equals the index of the subgroup $\Gamma_0(p_1^{n_1}p_2^{n_2}p_3^{n_3})$ so this function has zero of maximal order of vanishing at the cusp $\infty$.
Let us check the other conditions:
\begin{align*}
\sum\limits_{\delta|N}\delta' r_\delta&=r(p_1p_2p_3+p_1p_2p_3+p_1p_2p_3+p_1p_2p_3-p_1p_2p_3-p_1p_2p_3-p_1p_2p_3-p_1p_2p_3)=0\\
\sum\limits_{\delta|N}\delta r_\delta&=r(p_1+p_2+p_3+p_1p_2p_3-1-p_1p_2-p_2p_3-p_1p_3)=\\
                                     &=\frac{24}{(p_1-1)(p_2-1)(p_3-1)}(p_1-1)(p_2-1)(p_3-1)=24\\
\prod\limits_{\delta|N} \delta'^{r_\delta}&=(p_2p_3)^{p_1r}(p_1p_3)^{p_2r}(p_1p_2)^{p_3r}(p_1p_2p_3)^{-r}p_3^{-p_1p_2r}p_2^{-p_1p_3r}p_1^{-p_2p_3r}\\
                                          &=p_1^{r(p_2+p_3-1-p_2p_3)}p_2^{r(p_1+p_3-1-p_1p_3)}p_3^{r(p_1+p_2-1-p_1p_1)}\\
                                          &=p_1^{\frac{-24}{(p_1-1)(p_2-1)(p_3-1)}(p_2-1)(p_3-1)}p_2^{\frac{-24}{(p_1-1)(p_2-1)(p_3-1)}(p_1-1)(p_3-1)}p_3^{\frac{-24}{(p_1-1)(p_2-1)(p_3-1)}(p_1-1)(p_2-1)}\\
                                          &=p_1^{\frac{-24}{(p_1-1)}}p_2^{\frac{-24}{(p_2-1)}}p_3^{\frac{-24}{(p_3-1)}}.
\end{align*}
We have $p_1,p_2,p_3\in\left\lbrace 2,3,5,7,13\right\rbrace$, so the numbers $\frac{-24}{p_1-1}$, $\frac{-24}{p_2-1}$ and $\frac{24}{(p_3-1)}$  are even. It follows that $\prod\limits_{\delta|N} \delta'^{r_\delta}$ is a square of a rational number.

We have proved that $\Delta_{p_1^{n_1}p_2^{n_2}p_3^{n_3},12}(z)\in M_{12}(\Gamma_0(p_1^{n_1}p_2^{n_2}p_3^{n_3}))$ and has maximal order of vanishing at the cusp $\infty$.

\end{proof} 
 
Now we use these functions to construct maps $X_0(N)\to\mathbb{P}^2$. The divisor of the function $\Delta_{N,12}$ from Theorem \ref{NN} with respect to the group $\Gamma_0(N)$ is 
\begin{equation}\label{e1}
\text{div}(\Delta_{N,12})=\left[\Gamma(1):\Gamma_0(N) \right]\textbf{a}_\infty= N\prod_{p|N}(1+1/p)\textbf{a}_\infty.
\end{equation}
The divisors of functions $\Delta$, $\Delta(N\cdot)$ with respect to $\Gamma_0(N)$ are given by (\cite{Muic2}, Lemma 4-3)

\begin{align}\label{e2}
&\text{div}(\Delta)=\sum_{c/d \in \mathcal{C}_N}\frac{N}{d(d,N/d)}\textbf{a}_{\frac{c}{d}},\\
&\text{div}(\Delta_N)=\sum_{c/d \in \mathcal{C}_N}\frac{d}{(d,N/d)}\textbf{a}_{\frac{c}{d}}.\nonumber
\end{align}

\begin{theorem}\label{potpro}
Assume $N>1$ belongs to the set $S_1$ from Theorem \ref{NN} i.e., $N$ has the form $p^n$ for $p\in\lbrace 2,3,5,7,13\rbrace$. Let $\Delta_{N,12}$ be the eta-quotient from Theorem \ref{NN}. The modular curve $X_0(N)$ is birationally equivalent with the curve $\mathcal{C}(\Delta_{N,12},\Delta,\Delta_N) \subseteq\mathbb{P}^2$ whose degree is $p^{n-1}(p+1)-1$.
\end{theorem}
\begin{proof}
From $(\ref{e1})$ and $(\ref{e2})$ we have 
\begin{align*}
&\deg\left(\text{div}_\infty\left(\frac{\Delta}{\Delta_{p^n,12}}\right)\right)=p^{n-1}(p+1)-1\\
&\deg\left(\text{div}_\infty\left(\frac{\Delta_p^n}{\Delta_{p^n,12}}\right)\right)=p^{n-1}(p+1-p)=p^{n-1}.
\end{align*}
Divisors of the second number are powers of $p$ and these numbers don't divide the first number. So these two numbers are relatively prime for $n>1$. 

From Lemma \ref{stupnjevipolovi} it follows that the map $X_0(N)\to \mathbb{P}^2$ given by 
\begin{equation}\label{e3}
\mathfrak{a}_z\mapsto(\Delta_{N,12}(z):\Delta(z):\Delta(Nz))
\end{equation}
is a birational equivalence. 
Minumum of divisors $(\ref{e1})$,$(\ref{e2})$ equals $1$, so from the formula 
 $(\ref{formulazastupanj})$ we can calculate the degree of the image curve and it equals $\dim M_{12}(\Gamma_0(N))+g(\Gamma_0(N))-1-1$. Formula  $(\ref{racun})$ implies this number equals $p^{n-1}(p+1)-1$.

\end{proof}

Table $\ref{tab1}$ contains defining polynomials for the image curves $\mathcal{C}(\Delta_{N,12},\Delta,\Delta(N\cdot)$ for some values of $N$.

\begin{minipage}{\linewidth}
\centering
\captionof{table}{Equations for curves $\mathcal{C}(\Delta_{N,12},\Delta,\Delta(N\cdot)$ from Theorem \ref{potpro}} \label{tab1} 
\begin{tabular}{l| l}
N=2& $x_{0} x_{1} -  x_{2}^{2}$\\ \hline

N=3 & $x_{0}^{2} x_{1} -  x_{2}^{3}$\\ \hline
N=4& $x_{0}^{3} x_{1}^{2} + 4096 x_{0}^{3} x_{1} x_{2} + 48 x_{0}^{2} x_{1} x_{2}^{2} -  x_{2}^{5}$\\ \hline

N=5 &$x_{0}^{4} x_{1} -  x_{2}^{5}$\\ \hline

N=7 &$x_{0}^{6} x_{1} -  x_{2}^{7}$\\   \hline

N=9  & \makecell{$x_{0}^{8} x_{1}^{3} + 531441 x_{0}^{8} x_{1}^{2} x_{2} + 282429536481 x_{0}^{8} x_{1} x_{2}^{2} + 27894275208 x_{0}^{7} x_{1} x_{2}^{3}- 756 x_{0}^{6} x_{1}^{2} x_{2}^{3}$\\$  + 975725676 x_{0}^{6} x_{1} x_{2}^{4} + 14171760 x_{0}^{5} x_{1} x_{2}^{5} + 74358 x_{0}^{4} x_{1} x_{2}^{6} + 72 x_{0}^{3} x_{1} x_{2}^{7} -  x_{2}^{11}$}\\   \hline
N=13& $x_{0}^{12} x_{1} -  x_{2}^{13}$\\   \hline
\end{tabular}
\end{minipage}
\bigskip
\begin{theorem}
Assume $N$ has the form $2^n3^m$, $2^n5^m$, $2^n 13^m$ or $3^n5^m$. Let $\Delta_{N,12}$ be the eta-quotient from Theorem \ref{NN}. Modular curve $X_0(N)$ is birationally equivalent to the curve $\mathcal{C}(\Delta_{N,12},\Delta,\Delta_N) \subseteq\mathbb{P}^2$ which has degree equal to $\dim M_{12}(\Gamma_0(N))+g(\Gamma_0(N))-2$.

\end{theorem}
\begin{proof}
From $(\ref {e1})$ and $(\ref{e2})$ it follows that 

\begin{align*}
&\deg\left(\text{div}_\infty\left(\frac{\Delta}{\Delta_{N,12}}\right)\right)=p^{n-1}q^{m-1}(p+1)(q+1)-1\\
&\deg\left(\text{div}_\infty\left(\frac{\Delta_N}{\Delta_{N,12}}\right)\right)=p^{n-1}q^{m-1}((p+1)(q+1)-pq)=p^{n-1}q^{m-1}(p+q+1).
\end{align*}
These numbers are relatively prime if numbers $(p+q+1)$ and $ p^{n-1}q^{m-1}(p+1)(q+1)-1$ are relatively prime.
We check all possible cases:
\begin{itemize}
\item Let $N=2^n3^m$.
Then $2+3+1=6$ and $2$ and $3$ are prime divisors of $6$. But these two numbers don't divide $2^{n-1}\cdot 3^{m-1}\cdot 3\cdot 4-1$.
\item Let $N=2^n5^m$. Then
$2+5+1=8$ has one prime divisor $2$ which does not divide $2^{n-1}5^{m-1}\cdot 3\cdot 6-1$.

\item Let $N=2^n 13^m$. Then
$2+13+1=16$ has one prime divisor $2$ which does not divide $2^{n-1}\cdot 13^{m-1}\cdot 3\cdot 14-1$.
\item Let $N=3^n5^m$. Then
 $3+5+1=9$ has one prime divisor $3$ which does not divide $3^{n-1}\cdot 5^{m-1}\cdot 4\cdot  6-1$.

\end{itemize}
From Lemma \ref{stupnjevipolovi} we have birational equivalence of $X_0(N)$ and $C(\Delta_{N,12},\Delta,\Delta_N)$.

\end{proof}

\begin{theorem}
Let $N=2^{n_1}3^{n_2}7^{n_3}$ for $n_1,n_2,n_3\geq 1$. Let $\Delta_{N,12}$ be the eta-quotient from Theorem \ref{NN}. Modular curve $X_0(N)$ is birationally equivalent to the curve $C(\Delta_{N,12},\Delta,\Delta_N) \subseteq\mathbb{P}^2$ which has degree equal to $\dim M_{12}(\Gamma_0(N))+g(\Gamma_0(N))-2$.

\end{theorem}
\begin{proof}
From $(\ref{e1})$ and $(\ref{e2})$ we have 
\begin{align*}
&\deg\left(\text{div}_\infty\left(\frac{\Delta}{\Delta_{N,12}}\right)\right)=2^{n_1-1}3^{n_2-1}7^{n_3-1}\cdot 3\cdot 4\cdot 8-1\\
&\deg\left(\text{div}_\infty\left(\frac{\Delta_N}{\Delta_{N,12}}\right)\right)=2^{n_1-1}3^{n_2-1}7^{n_3-1}( 3\cdot 4\cdot 8-2\cdot 3\cdot 7)=2^{n_1-1}3^{n_2-1}7^{n_3-1}\cdot 54.
\end{align*}

Prime divisors of  $\deg\left( \text{div}_\infty\left(\frac{\Delta_N}{\Delta_{N,12}}\right)\right)$ belong to the set $\left\lbrace 2,3 ,7 \right\rbrace $. None of the numbers from this set divides the number $\deg\left(\text{div}_\infty\left(\frac{\Delta}{\Delta_{N,12}}\right)\right)$. These numbers are relatively prime, so by Lemma \ref{stupnjevipolovi} we conclude that the map defined by functions $\Delta_{N,12},\Delta$ and $\Delta_N$ is birational equivalence.
\end{proof}

It is our conjecture that the map $$\mathfrak{a}_z\mapsto(\Delta_{N,12}(z):\Delta(z):\Delta(Nz))$$ is birational equivalence for all functions $\Delta_{N,12}$ from Theorem \ref{NN} but the argument with divisors of poles is satisfied only in the mentioned cases. In other cases, divisors of poles of used functions are not relatively prime. 

As an example, for $N=2^37^1=56$ we have 

 \begin{align*}
 &\deg\left(\text{div}_\infty\left(\frac{\Delta}{\Delta_{56,12}}\right)\right)=95\\
 &\deg\left(\text{div}_\infty\left(\frac{\Delta_{56}}{\Delta_{56,12}}\right)\right)=40.
 \end{align*}
However, we have developed an algorithm that calculates the degree of the resulting curve and with the aid of formula $(\ref{formulazastupanj})$ we calculated that in this case the degree of the map equals $1$.

\end{document}